\documentclass[11pt,reqno]{amsart}
\usepackage{graphicx}
\pdfoutput=1
\usepackage{amsfonts,amsmath,amssymb}
\usepackage{hyperref}
\usepackage{paralist}
\usepackage[linesnumbered,ruled,vlined,norelsize]{algorithm2e}

\usepackage{thmtools}
\declaretheoremstyle[bodyfont=\normalfont]{noncursive}
\declaretheorem{theorem}
\declaretheorem[numberwithin=section]{lemma}

\declaretheorem[numberlike=lemma]{proposition}
\declaretheorem[numberlike=lemma]{corollary}
\declaretheorem[style=noncursive,numberlike=lemma]{definition}

\declaretheorem[style=noncursive,numberlike=lemma]{remark}

\newcommand{\im}{\ensuremath{\mbox{\rm Im}\,}}
\newcommand{\re}{\ensuremath{\mbox{\rm Re}\,}}

\newcommand{\CC}[1]{\mathbb{C}^{#1}}

\newcommand{\RR}[1]{\mathbb{R}^{#1}}
\newcommand{\dw}{\frac{\partial}{\partial w}}

\newcommand{\dz}{\frac{\partial}{\partial z}}

\newcommand{\lr}{\longrightarrow}

\numberwithin{equation}{section}

\newcommand{\hol}[1]{\mathfrak{hol}^{#1}}
\newcommand{\aut}[1]{\mathfrak{aut}^{#1}}

\title[New extension phenomena]{New extension phenomena for solutions of tangential Cauchy\,-\,Riemann Equations}

\author {I. Kossovskiy}
\address{Department of Mathematics, University of Vienna}
\email{ilya.kossovskiy@univie.ac.at}
\author {B. Lamel}
\address{Department of Mathematics, University of Vienna}
\email{bernhard.lamel@univie.ac.at}

\begin{document}

\maketitle

\date{\today}

\begin{abstract}
In our recent work \cite{nonanalytic} we showed that $C^\infty$ CR-diffeomorphisms of real-analytic Levi-nonflat hypersurfaces in $\CC{2}$ are not analytic in general. This result raised again the question on the nature of CR-maps of real-analytic hypersurfaces. 

In this paper, we give a complete picture of what CR-maps actually are. First, we discover an analytic continuation phenomenon for CR-diffeomorphisms which we call the {\em sectorial analyticity property}. It appears to be the optimal regularity property for CR-diffeomorphisms in general. We emphasize that such type of extension never appeared previously in the literature. Second, we introduce the  class of {\em Fuchsian type hypersurfaces} and prove that (infinitesimal generators of) CR-automorphisms of a Fuchsian type hypersurface are still analytic. In particular, this solves a problem formulated in \cite{nonminimalODE}. 

Finally, we prove a regularity result for {\em formal} CR-automorphisms of Fuchsian type hypersurfaces.
\end{abstract}

\tableofcontents

\section{Introduction}

\subsection{Overview}

The problem of regularity of CR-maps between CR-submanifolds in complex space is of fundamental importance in the field of Several Complex Variables. Starting from the classical work of Cartan \cite{cartan},  Chern and Moser \cite{chern}, Pinchuk \cite{pinchuksib}, and Lewy \cite{lewy}, a large amount of publications is dedicated to various positive results in this well developed  direction. In particular, when both the source and the target are real-analytic, the expected regularity of smooth CR-maps is $C^{\omega}$, i.e., they are {\em analytic} (this property implies that the CR-maps extend holomorphically to a neighborhood of the source manifold).  We refer the reader to Jacobowitz \cite{jacobowitz},  Baouendi, Ebenfelt and Rothschild \cite{ber}, Forstneri\'c \cite{forstneric},  Berhano, Cordaro and Hounie \cite{cordaro}, and the Introduction in \cite{nonanalytic} for the set-up of the theory of CR-maps, a historic outline of the analyticity problem, its connections with the boundary regularity of holomorphic maps\,/\, the reflection principle, and the connections of the problem to  the theory of Linear PDEs.  

However,  in the recent paper \cite{nonanalytic} the authors discovered the existence of real-analytic hypersurfaces in $\CC{N},\,N\geq 2$ which are $C^\infty$ CR-equivalent, but are inequivalent analytically. In particular, it follows that $C^\infty$ CR-diffeomorphisms between real-analytic Levi-nonflat hypersurfaces in $\CC{2}$ are {\em not} analytic in general. Moreover, it shows that the equivalence problem for nonminimal real-analytic CR-structures is of a more {\em intrinsic} nature, as a map realizing an equivalence does not necessarily arise from the biholomorphic equivalence of the CR-manifolds as submanifolds in complex space.

Two natural questions are immediately raised by the results in \cite{nonanalytic}. First, we shall ask what the actual nature of CR-maps is (that is, we search for an optimal property of smooth CR-diffeomorphisms replacing the analyticity). Second, we shall identify an optimal class of ``regular'' real-analytic hypersurfaces, for which CR-diffeomorphism are still analytic. 

The goal of the current paper is to answer the above two questions in the $\CC{2}$-case  in large generality. 

For the first question, we discover the phenomenon of {\em sectorial analyticity} of CR-diffeomorphisms, that is, their holomorphic extension to so-called {\em tangential sectorial domains} with an asymptotic power series representation in such a domain. We shall emphasize that this gives a very new type of extension for solutions of tangential Cauchy-Riemann equations, which never appeared previously in the literature. This phenomenon somehow parallels the very familiar phenomenon of {\em wedge extendability} for CR-functions (see Baracco, Zampieri and Zaitsev \cite{bzz} for most recent results in this direction).  

Comparing with the construction in \cite{nonanalytic}, we see that the sectorial analyticity property can not be strengthened further in general. Thus, {\em we obtain a complete picture of what CR-maps actually are}. 

For the second question, we introduce the class of {\em Fuchsian type hypersurfaces} (the latter condition is described explicitly in terms of the defining function of a hypersurface), for which CR-diffeomorphisms are still analytic. We also show optimality of the Fuchsian type condition. By that, we obtain a solution to the problem formulated in the paper \cite{nonminimalODE} of Shafikov and the first author. 

Another result of us is connected to the problem of convergence of {\em formal} CR-maps. Similarly to the analyticity issue, this problem has attracted a lot of attention of experts in Complex Analysis in the last few decades (see, e.g., the survey \cite{mirapprox} of Mir). \autoref{convergence} below establishes a convergence result for formal CR-maps in the Fuchsian type case.

\smallskip

We formulate the results below in detail. We start with describing the precise class of hypersurfaces considered in this paper.

\smallskip 

Let $M,M'\subset\CC{2}$ be real-analytic Levi-nonflat hypersurfaces passing through the origin. We recall that due to the work of many authors the problem of analyticity of CR-maps between such hypersurfaces   is well understood when both the source and the target are  {\em minimal} in the sense of Tumanov \cite{tumanov}\,/\,{\em finite type}  in the sense of D'Angelo \cite{dangelo} at the origin (both properties in the $\CC{2}$\,-\,case amount to the non-existence of a non-trivial complex curve $X\subset M$ though the origin).  The analyticity is also established in the {\em algebraic} case due to Baouendi, Huang  and Rothschild \cite{bhr}. However,
  the problem of propagating the analyticity phenomenon to the class of {\em nonminimal} hypersurfaces (in the non-algebraic case) remained widely open. The nonminimal case is particularly difficult for the study of regularity of CR-mappings, and very little is known in this setting. The state of the art and open problems here are probably best addressed in the survey \cite{hu2} of Huang. 
  
  Accordingly, we are concerned here with the nonminimal (infinite type) case.   The (locally unique!) complex curve, contained in $M$ and passing through the origin, is denoted by $X$ in what follows. Recall that $X$ is necessarily smooth. We also use, {\em locally near $p$}, the decomposition 
\begin{equation}\label{decompose}
\Sigma=X\cup\Sigma_1
\end{equation}
for the Levi degeneracy set $\Sigma$ of $M$  (here $\Sigma_1$ is a component of the real-analytic set $\Sigma$, not containing nontrivial complex germs). The fact that  $(M\setminus X)\cap U$ is Levi-nondegenerate for some neighborhood $U$ of $p$ in $\CC{2}$ reads then as $\Sigma_1=\emptyset$. 

\smallskip

 Throughout this paper, we restrict the considerations to {\em CR-automorphisms} having an infinitesimal generator (i.e., we stay within the category of flows of infinitesimal CR-automorphisms). 
 We do this restriction in order to give a simpler exposition of the new phenomena discovered in the paper, and to emphasize the elegance of the method used here.  
 
 We recall   that the {\em infinitesimal
automorphism  algebra }  for a
real submanifold $M\subset\CC{N}$ at a point $p\in M$ \rm is the
algebra $\hol{k} (M,0)$
of holomorphic ($k=\omega$)  or smooth ($k=\infty$) vector fields
\[X=f_1\frac{\partial}{\partial z_1}+...+f_n\frac{\partial}{\partial z_N},\] defined near $p$
such that each $f_j$ is a real-analytic ($k=\omega$) or smooth ($k=\infty$)
CR-function on $M$ and
 $X+\bar X$ is tangent to $M$ near $p$. Vector fields $X\in\mathfrak{hol}\,(M,0)$ (resp.
$X\in\mathfrak{hol}^\infty(M,0)$) are exactly the vector fields
generating flows of holomorphic (resp. smooth CR)
transformations, preserving $M$ locally. The  \emph{stability
subalgebras}
$\aut{k}(M,0)\subset\hol{k} (M,0)$ are
 determined by the condition $X|_p=0$.

\subsection{The sectorial extension}

Let $M\subset\CC{2}$ be a real-analytic  Levi nonflat hypersurface, which is  nonminimal at a point $p\in M$, and $X\ni p$ its complex locus.  We recall that (e.g. \cite{lmblowups}), for any real-analytic curve $\gamma\subset M$ passing trough $p$ and transverse to the complex tangent $T^{\CC{}}_p M$, we can  choose local holomorphic coordinates near $p$ such that in these coordinates $p$ is the origin, the complex locus is given by
\begin{equation}\label{locus}
X=\{w=0\},
\end{equation} 
and the curve $\gamma$ is described as
\begin{equation}\label{Gamma}
\Gamma=\bigl\{z=0,\im w=0\bigr\}.
\end{equation} 
We next introduce
\begin{definition}\label{tangential} A set $D_p\subset\CC{2},\,D_p\ni p$ is called a \em tangential sectorial domain for $M$ at $p$  \rm if, in some   local holomorphic coordinates   $(z,w)$ for $M$ as above, the set $D_p$ looks as 
\begin{equation}\label{standard}
\Delta\times \Bigl(S^+\cup\{0\}\cup S^-\Bigr).
\end{equation}
Here $\Delta\subset\CC{}$ is a disc of radius $r>0$, centered at the origin, and $S^\pm\subset\CC{}$ are sectors
\begin{equation}\label{sectors}
S^+=\bigl\{|w|<R,\, -\alpha<\mbox{arg}\,w<\alpha\bigr\},\quad S^-=\bigl\{|w|<R,\, \pi-\alpha<\mbox{arg}\,w<\pi+\alpha\bigr\}
\end{equation}
\noindent for appropriate $R>0,\,0<\alpha<\pi/2$ and $\mbox{arg}\,w\in(-\pi/2,3\pi/2]$. 
We also denote by $D^\pm_p$ the domains $\Delta\times S^\pm\subset\CC{2}$ respectively.
\end{definition} 
We further note that  in the above local holomorphic coordinates the hypersurface $M$ is given by an equation
$$\im w=(\re w)^m\psi(z,\bar z,\re w),\quad m\geq 1,\quad \psi(0,0,\re w)=0,$$
which implies that 

\medskip

{\em for any tangential sectorial domain $D_p$ for $M$ at $p$,  the intersection of $M$ with a sufficiently small neighborhood $U_p$ of $p$ in $\CC{2}$ is contained in $D_p$}.

   \medskip
   
 \noindent  We now give the following   
\begin{definition}\label{asymptotic}
 We say that 
 a $C^\infty$ CR-function $f$ in a neighborhood of $p$ in $M$ is {\em sectorially extendable,} if for some (and then any sufficiently small) tangential sectorial domain $D_p$ for $M$ at $p$, there exist functions $f^\pm\in\mathcal O(D^\pm_p)$ such that 
 
 \smallskip
 
 $(i)$ each $f^\pm$ coincides with $f$ on $D^\pm_p\cap M$, and 
 
 \smallskip
 
$ (ii)$ both $f^\pm$ admit the same {\em asymptotic representation}
 $$ f^\pm\sim \sum_{k,l\geq 0}a_{kl}z^kw^l$$
 in the respective domains $D^\pm_p$. Recall that the latter property means that for all $N\geq 0$ we have 
 $$f^\pm(z,w)-\sum_{k+l=N}a_{kl}z^kw^l=o\left(|z|^N+|w|^N\right),\,\,\mbox{when}\,\,\, (z,w)\in D_p^\pm,\,(z,w)\rightarrow (0,0);$$
  see also \cite{vazow} for more details of this concept. 
  
  \end{definition}
 
 \smallskip

 We can similarly define the sectorial extendability of CR-mappings or infinitesimal CR-automorphisms of real-analytic hypersurfaces. We show in Section 3 that the notions of a tangential sectorial domain and the sectorial extendability, respectively, are in a sense coordinates-free. We also show that the property of ``being small'' in Definition 1.2 applies to the radii of the disc and the sector only, but not to the angle of the sector.

 \smallskip

We now formulate our sectorial extension result. Recall that by $\Sigma_1$ we denote the additional component of the Levi degeneracy set $\Sigma$ (if the latter exists), as in \eqref{decompose}.

\begin{theorem}\label{sectorial}
Let $M\subset\CC{2}$ be a real-analytic Levi-nonflat hypersurface, which is nonminimal at a point $p$. 
Assume, in addition, that $M$ is Levi-nondegenerate in the complement of its complex locus, i.e. $\Sigma_1=\emptyset$. Then any infinitesimal automorphism $L=P\dz+Q\dw\in \hol{\infty}(M,p)$ is  sectorially extendable.  
\end{theorem}

The assertion of \autoref{sectorial} can be immediately reformulated for the class of CR-automorphisms which have an infinitesimal generator $L$.

\subsection{Analyticity in the Fuchsian type case} As was discussed above, we introduce now the class of {\em Fuchsian type} hypersurfaces for which CR-diffeomorphisms demonstrate regular behaviour. To formulate the Fuchsian condition, we need to use special coordinates \eqref{madmissible} (called {\em $m$-admissible coordinates}; we refer the reader to Section 3 for details) which exist only in the case when $M\setminus X$ is Levi-nondegenerate, i.e., when $\Sigma_1=\emptyset$ in \eqref{decompose}. Such a transfer is possible due to the work \cite{nonminimalODE} of Shafikov and the first author. In order to deal with the general case, we need to use the resolution-of-degeneracies construction introduced in \cite{lmblowups} by Mir and the second author.

\begin{definition} 
Let $M\subset\CC{2}$ be a real-analytic Levi-nonflat hypersurface, $p\in M$. We say that a  real-analytic hypersurface $M\rq{}\subset\CC{2}_{(\xi,\eta)}$ containing the germ at the origin of the complex line $X=\{\eta=0\}$ is {\em a monomial blow-up of $M$ at $p$}, if for some local holomorphic coordinates near $p$ at which $p$ is the origin there exists a blow-down map of the form 
\begin{equation}\label{blowup}
F(\xi,\eta):\quad (\CC{2},0)\lr (\CC{2},0),\quad F(\xi,\eta)=(\xi\eta^s,\eta^l),\quad s,l\in\mathbb{Z},\quad s,r\geq 1,
\end{equation}
such that $F(M\rq{})\subset M$. One necessarily has  $F(X)=\{0\}$.
\end{definition}

As the result in \cite{lmblowups} shows, {\em any} Levi-nonflat hypersurface admits a monomial blow-up with $s\geq 2,\,l=2$ in appropriate local holomorphic coordinates, which has the property $\Sigma_1=\emptyset$ in \eqref{decompose} (see Section 2 for details). 

We are now in the position to formulate the Fuchsian type condition.

\begin{definition} \label{Fuchsian}
Let $M\subset\CC{2}$ be a real-analytic Levi-nonflat hypersurface, containing a complex hypersurface $X$. Assume first that $\Sigma_1=\emptyset$. We then say that $M$ is {\em of Fuchsian type at $q$}, if in some local holomorphic coordinates near $q$ at which  $q$ is the origin and $M$ is given by \eqref{madmissiblereal}, the functions $h_{kl}$, as in \eqref{madmissiblereal}, satisfy
\begin{equation}\label{fuchsianity}
\begin{aligned}
\mbox{ord}_0\,h_{22}(u)&\geq m-1,\quad \mbox{ord}_0\, h_{23}(u)\geq 2m-2,\quad \mbox{ord}_0\, h_{33}(u)\geq 2m-2,\\ 
\mbox{ord}_0\,
h_{24}(u)&\geq 3m-3, \quad \mbox{ord}_0\, h_{34}(u)\geq 3m-3.
\end{aligned}
\end{equation}
If, otherwise, $\Sigma_1\neq\emptyset$, we say that $M$ is {\em Fuchsian} at $q$ if there exists a monimial blow-up of it with $\Sigma_1=\emptyset$ which is Fuchsian at the origin.
\end{definition}  We then have an important

\begin{remark}\label{mequals1}
We emphasize that the Fuchsian condition holds automatically for $m=1$ and does {\em not} hold in general for $m>1$ (see, e.g., examples provided in \cite{nonminimalODE}).
\end{remark}

\begin{remark}\label{comparison}
The Fuchsian condition here should be compared with the one introduced by Shafikov and the first author in \cite{nonminimalODE} for the special class of nonminimal hypersurfaces, {\em Levi-nondegenerate and spherical in $M\setminus X$}. We emphasize that our Fuchsian condition is slightly stronger than the one  introduced in \cite{nonminimalODE}. In particular, it includes conditions on the functions $h_{24},h_{34}$ which are not apparent in  \cite{nonminimalODE} due to the { sphericity} of hypersurfaces there. Following the proof in Section 4, the reader can see that in the general (non-spherical) case the latter conditions on  $h_{24},h_{34}$ can {\em not} be avoided.
\end{remark}

We now formulate our analyticity result.

\begin{theorem}\label{analyticity}
Let $M\subset\CC{2}$ be a real-analytic Levi-nonflat hypersurface, which is nonminimal at a point $p$. Suppose $p$ is a Fuchsian type point. Then any vector field $L\in\hol{\infty}(M,p)$  extends holomorphically to a neighborhood of $p$.
\end{theorem}

We shall note here that for $m=1$ the analyticity property for CR-diffeomorphisms was already established in the work of Ebenfelt \cite{ebenfelt}. Thus, in view of  \autoref{mequals1}, \autoref{analyticity} extends the result of Ebenfelt for $m>1$. We shall also note that for non-Fuchsian hypersurfaces the assertion of \autoref{analyticity} does {\em not} hold in general, as the construction in \cite{nonanalytic} shows. Thus, the result in \autoref{analyticity} is in a sense optimal.

  Applying the Hanges-Treves propagation principle \cite{ht}, we immediately obtain 

\begin{corollary}\label{anypoint}
Let $M\subset\CC{2}$ be a real-analytic Levi-nonflat hypersurface, which is nonminimal at a point $p$, and $X\ni p$ the complex hypersurface, contained in $M$. Assume that $L\in\hol{\infty}(M,p)$ is defined at some Fuchsian type point $q\in X$. Then $L$ extends holomorphically to a neighborhood of $p$.
\end{corollary}

We also formulate below a curious intermediate result on the analytic continuation of infinitesimal CR-automorhisms, defined {\em in the complement} of the complex locus $X$.

\begin{theorem}\label{extension}
Let $M\subset\CC{2}$ be a real-analytic Levi-nonflat hypersurface, which is nonminimal at a point $p$. Let $U$ be a sufficiently small neighborhood of $p$ in $\CC{2}$, and $X\ni p$ the complex hypersurface, contained in $M$. 

\smallskip

(i) Suppose that $(M\setminus X)\cap U$ is Levi-nondegenerate. Then, for any $q\in (M\setminus X)\cap U$ and  any infinitesimal automorphism $L$ of $M$ at $q$, the germ of $L$ at $q$  extends analytically along any path $\gamma\subset U\setminus X$, starting at $q$.

\smallskip 

(ii) Suppose, otherwise, that $(M\setminus X)\cap U$ contains Levi-degenerate points.  Then there exists local holomorphic coordinates $(z,w)$ near $p$ at which $p$ is the origin and $X=\{w=0\}$, and a blow-down map 
\begin{equation}\label{lmblowup}
F(\xi,\eta):\quad (\CC{2},0)\lr (\CC{2},0),\quad F(\xi,\eta)=(\xi\eta^s,\eta^2),\quad s\in\mathbb{Z},\quad s\geq 2,
\end{equation}
such that for a sufficiently small neighborhood $U\rq{}$ of the origin in $\CC{2}_{(\xi,\eta)}$, any point $q\in (M\setminus X)\cap F(U\rq{})$, and any infinitesimal automorphism $L$ of $M$ near $q$,    the pulled-back germ $(F^{-1})^*(L)$ at $q\rq{}:=F^{-1}(q)$ extends analytically along any path $\gamma\rq{}\subset U\rq{}\setminus \{\eta=0\}$, starting at $q\rq{}$, for an appropriate choice of a branch of $F^{-1}$. 
\end{theorem}

\autoref{extension} should be compared with the extension-along-path result of Shafikov and the first author obtained in \cite{nonminimal}.

\subsection{Regularity of formal CR-automorphisms} We recall that  convergence of formal CR-diffeomorphisms between real-analytic finite type hypersurfaces  in $\CC{2}$ holds due to the result of Baouendi, Ebenfelt and Rothschild \cite{ber1}. In the nonminimal case, the only known result is due to Juhlin and the second author \cite{jl2} (they prove convergence in the case $m=1$ in \eqref{mnonminimal}). 

Somewhat surprisingly for the field, in the paper \cite{divergence} Shafikov and the first author constructed examples of real-analytic hypersurfaces in $\CC{2}$ which are formally but not holomorphically equivalent. In particular, this implies the divergence of  formal CR-diffeomorphisms between nonminimal Levi-nonflat hypersurfaces in general. Thus the problem of distinguishing the optimal class of hypersurfaces with the convergence property came out naturally.   

The following theorem below extends the convergence result in \cite{jl2}.

\begin{theorem}\label{convergence}
Let $M\subset\CC{2}$ be a real-analytic Levi-nonflat hypersurface, which is nonminimal at a point $p$. Suppose $p$ is a Fuchsian type point. Then any formal infinitesimal automorphism $L\in\hol{f}(M,p)$  is convergent.
\end{theorem}

As the construction in Section 4 and the examples provided in \cite{divergence} show, the result in \autoref{convergence} is optimal. 

\begin{remark}\label{solve} We emphasize that \autoref{analyticity} and \autoref{convergence}  above give a solution to a problem formulated in \cite{nonminimalODE}, which is the general problem of distinguishing the right class of Fuchsian type hypersurfaces, as hypersurfaces with ``regular behaviour'' of CR-maps at the nonminimal locus.  
\end{remark}

\subsection{Principal method} The main tool of the paper is the recent
CR \,$\lr$\,DS (Cauchy-Riemann manifolds \,$\lr$\,\,Dynamical Systems) technique introduced by Shafikov and both 
authors in their earlier work \cite{divergence,nonminimalODE,nonanalytic}. The technique suggests to
replace a given CR-submanifold $M$ with a CR-degeneracy (such as
nonminimality) by an appropriate holomorphic dynamical system
$\mathcal E(M)$, and then study mappings of CR-submanifolds
accordingly.   The possibility to replace a real-analytic
CR-manifold by a complex dynamical system is based on the
fundamental parallel between CR-geometry and the geometry of
completely integrable PDE systems. This parallel was first observed by E.~Cartan and
Segre \cite{cartan,segre} (see also Webster \cite{webster}), and was revisited and developed in the important series of publications by
Sukhov \cite{sukhov1,sukhov2}. The ``mediator''
between a CR-manifold and the associated PDE system is the Segre
family of the CR-manifold. Unlike the nondegenerate setting in the cited work \cite{cartan,segre,sukhov1,sukhov2}, the CR\,-\,DS technique  deals systematically with the degenerate setting, providing sort of a dictionary between CR-geometry and Dynamical Systems.
We refer to Section 2 for further details and references and briefly mention here that, using the CR -- DS technique, we can interpret infinitesimal CR-automorphisms as {\em Lie symmetries} of the associated dynamical systems (singular differential equations). After that, the polynomiality of the jet prolongation of a vector field in the jet variables allows us to work out the tangency condition of a Lie symmetry with the differential equation, and then prove the desired properties of the solutions of the tangency equations by applying methods of the theory of complex ODEs with an isolated singularity.

\bigskip

\begin{center}\bf Acknowledgments \end{center}

\smallskip

The first author is fully,  and the second author is partially supported by the Austrian Science Fund (FWF). 

\smallskip

\section{Preliminaries}
\label{sec:prelim}
\subsection{Nonminimal real hypersurfaces} 
\label{sub:nonminimal_real_hypersurfaces}
We recall that given a real-analytic hypersurface $M\subset \CC{2}$, for 
every $p\in M$ there exist so-called {\em normal coordinates} $(z,w)$ centered
at $p$, i.e. a local holomorphic coordinate system near $p$ in 
which $p=0$ and near $0$, $M$ is defined by an equation of the form 
\[ v = F (z, \bar z , u)\] 
for some germ $F$ of a holomorphic function on $\CC{3}$ which satisfies 
\[ F(z,0,u) = F(0,\bar z, u) = 0\]
and the 
{\em reality condition} $F(z,\bar z,u) \in \RR{} $ for $(z,u)\in \CC{} \times\RR{}$ close to $0$ (see e.g. \cite{ber}). 

We say that $M$ is {\em nonminimal} at $p$ if there exists a germ of a nontrivial
complex curve $X\subset M$ through $p$. It turns out that in normal coordinates, such 
a curve $X$ is necessarily defined by $w = 0$; 
in particular, any such $X$ is nonsingular.

Thus $M$ is nonminimal if and only if with normal coordinates $(z,w)$ and a 
defining function $F$ as above, we have that $F(z,\bar z,0) = 0$, or equivalently, 
if $M$ can defined by an equation of the form 
\begin{equation}\label{mnonminimal}
v=u^m\psi(z,\bar z,u), 
\text{ with } \psi(z,0,u) = \psi(0,\bar z, u)= 0 \text{ and } \psi(z,\bar z,0)\not\equiv 0,
\end{equation}
where $m\geq 1$.
 
It turns out that the integer $m\geq 1$ is independent of the choice
of normal coordinates  (see \cite{meylan}), and actually also 
of the choice of $p\in X$; we refer to $m$ as the 
{\em nonminimality order} of $M$ on $X$ (or at $p$) and say that
$M$ is $m$-nonminimal along $X$ (or at $p$).

Several other variants of defining functions for $M$ are useful. The first one is 
the complex defining function $R$ in which $M$ is defined by 
\[  w = R (z,\bar z,\bar w) ;\]
it is obtained from $F$ by solving the equation 
\[ \frac{w - \bar w}{2i} = F \left(z,\bar z, \frac{w+\bar w}{2} \right) \]
for $w$. The complex defining function satisfies the conditions 
\[ R(z,0,\tau) = R (0, \chi, \tau ) = \tau, \quad R(z,\chi,\bar R (\chi, z, w)) = w. \]
If $M$ is $m$-nonminimal at $p$, then $R (z,\chi,\tau) = \tau \theta (z,\chi,\tau)$ and 
thus $M$ is defined by 
\[ w = \bar w \theta(z,\bar z,\bar w) = \bar w (1 +\bar w^{m-1}\tilde \theta(z,\bar z,\bar w) ), \text{ where } 
\tilde \theta (z,0,\tau) = \tilde \theta (0,\chi, \tau) = 0 \text{ and }  
\tilde\theta (z,\chi,0)\neq 0.\]
We can also rewrite 
the complex defining equation in its  {\em exponential form} 
\begin{equation}
\label{exponential}  w=\bar we^{i\varphi(z,\bar z,\bar w)} = \bar w e^{i \bar w^{m-1}\tilde\varphi (z,\bar z,\bar w)}, \text{ where } \tilde \varphi(z,0,\bar w) = \tilde \varphi(0,\bar z, \bar w) = 0 \text{ and } \tilde \varphi(z,\bar z,0)\not\equiv 0.
\end{equation}
 An  equation of the form \eqref{exponential} defines a real hypersurface if and 
 only if $\varphi$ satisfies the reality condition 
 \begin{equation}
\label{reality}
\varphi(z,\bar z,w\,e^{-i\bar\varphi(\bar z,z,w)})\equiv
\bar\varphi(\bar z,z,w).
\end{equation}

The {\em Segre family} of $M$, where $M$ is given in 
normal coordinates as above, with the complex defining function 
$R\colon U_z\times \bar U_z \times \bar U_w = U_z \times \bar U \to U_w$ consists of the complex hypersurfaces $Q_\zeta \subset U$, defined for  $\zeta \in U$
by 
\[ Q_\zeta = \{(z,w) \colon w = R (z,\bar \zeta)\}. \]
The real line $\Gamma = \{(z,w)\in M \colon z = 0 \} = \{(0,u) \in M \colon u\in\RR{}\}\subset M$
has the property that 
\[ Q_{(0,u)} = \{w = u\}, \quad (0,u)\in \Gamma \]
for $u\in\RR{}$, a property which actually is equivalent to the normality of 
the coordinates $(z,w)$. More exactly, for any real-analytic curve 
$\gamma$ through $p$ one can find normal coordinates $(z,w)$ in 
which $\gamma$ corresponds to $\Gamma$ (see e.g. \cite{lmblowups}). 

We shall use  a  technique to resolve (in a certain sense) degeneracies
of a real hypersurface $M$ introduced in \cite{lmblowups}:

\begin{lemma}[Blow-up Lemma]\label{blowuplemma}
Let $M\subset\CC{2}$ be a real-analytic hypersurface, which is Levi-degenerate at the origin and Levi-nonflat. Assume that  $M$ is given in normal coordinates \eqref{mnonminimal} and that the curve $\Gamma\subset M$, as in \eqref{Gamma}, does not contain Levi-degenerate points other than the origin. Then there exists  a blow-down map 
\begin{equation}\label{lmblowup}
F(\xi,\eta):\quad (\CC{2},0)\lr (\CC{2},0),\quad F(\xi,\eta)=(\xi\eta^s,\eta^2),\quad s\in\mathbb{Z},\quad s\geq 2,
\end{equation}
and a real-analytic nonminimal at the origin hypersurface $M^*\subset\CC{2}_{(\xi,\eta)}$ with the complex locus $X=\{\eta=0\}$ such that:

\smallskip

(i) $F(M^*)\subset M,\quad F(X)=\{0\}$;

\smallskip

(ii) $M^*\setminus X$ is Levi-nondegenerate.

\smallskip

(iii) $M^*$ is also given in normal coordinates.
\end{lemma}

\subsection{Real hypersurfaces and second order differential equations.}\label{sub:realhyp2ndorderequ}
To every Levi nondegenerate real hypersurface
$M\subset\CC{N}$  we can associate a system of second order
holomorphic PDEs with $1$ dependent and $N-1$ independent
variables, using the Segre family of the hypersurface. This remarkable construction
 goes back to
E.~Cartan \cite{cartanODE},\cite{cartan} and Segre \cite{segre} (see also \cite{webster}),
and was recently revisited in
\cite{sukhov1},\cite{sukhov2},\cite{nurowski},\cite{gm},\cite{divergence},\cite{nonminimalODE},\cite{nonanalytic} (see also
references therein). 
We  describe this procedure in the case
$N=2$ relevant for our purposes. For the details of the theory of Segre varieties we refer to \cite{ber},\cite{DiPi}.

Let $M\subset\CC{2}$ be a smooth real-analytic
hypersurface, passing through the origin, and $U = U_z \times U_w$
 a sufficiently small neighborhood of the origin. In this case
we associate a second order holomorphic ODE to $M$, which is uniquely determined by the condition that the equation is satisfied by all the
graphing functions $h(z,\zeta) = w(z)$ of the
Segre family $\{Q_\zeta\}_{\zeta\in U}$ of $M$ in a
neighbourhood of the origin.

More precisely, since $M$ is Levi-nondegenerate
near the origin, the Segre map
$\zeta\lr Q_\zeta$ is injective and the Segre family has
the so-called transversality property: if two distinct Segre
varieties intersect at a point $q\in U$, then their intersection
at $q$ is transverse. Thus, $\{Q_\zeta\}_{\zeta\in U}$ is a
2-parameter  family of holomorphic
curves in $U$ with the transversality property, depending
holomorphically on $\bar\zeta$. It follows from
the holomorphic version of the fundamental ODE theorem (see, e.g.,
\cite{ilyashenko}) that there exists a unique second order
holomorphic ODE $w''=\Phi(z,w,w')$, satisfied by all the graphing functions of
$\{Q_\zeta\}_{\zeta\in U}$.

To be more explicit we consider the complex defining equation $w=R(z,\bar z,\bar w)$, as introduced above. 
 The Segre
variety $Q_\zeta$ of a point $\zeta=(a,b)\in U$ is  now given
as the graph
\begin{equation} \label{segre0}w (z)=R(z,\bar a,\bar b). \end{equation}
Differentiating \eqref{segre} once, we obtain
\begin{equation}\label{segreder} w'=\rho_z(z,\bar a,\bar b). \end{equation}
Considering \eqref{segre0} and \eqref{segreder}  as a holomorphic
system of equations with the unknowns $\bar a,\bar b$, an
application of the implicit function theorem yields holomorphic functions
 $A, B$ such that
$$
\bar a=A(z,w,w'),\,\bar b=B(z,w,w').
$$
The implicit function theorem applies here because the
Jacobian of the system coincides with the Levi determinant of $M$
for $(z,w)\in M$ (\cite{ber}). Differentiating \eqref{segreder} once more
and substituting for $\bar a,\bar b$ finally
yields
\begin{equation}\label{segreder2}
w''=\rho_{zz}(z,A(z,w,w'),B(z,w,w'))=:\Phi(z,w,w').
\end{equation}
Now \eqref{segreder2} is the desired holomorphic second order ODE
$\mathcal E = \mathcal{E}(M) $.

More generally, the association of   a completely integrable PDE  with
a CR-manifold is possible for a wide range of
CR-submanifolds (see \cite{sukhov1,sukhov2,gm}). The
correspondence $M\lr \mathcal E(M)$ has the following fundamental
properties:

\begin{enumerate}

\item[(1)] Every local holomorphic equivalence $F:\, (M,0)\lr (M',0)$
between CR-submanifolds is an equivalence between the
corresponding PDE systems $\mathcal E(M),\mathcal E(M')$ (see \autoref{sub:equiv2ndorder});

\medskip

\item[(2)] The complexification of the infinitesimal automorphism algebra
$\mathfrak{hol}^\omega(M,0)$ of $M$ at the origin coincides with the Lie
symmetry algebra  of the associated PDE system $\mathcal E(M)$
(see, e.g., \cite{olver} for the details of the concept).

\end{enumerate}

Even though for a real hypersurface
$M\subset\CC{2}$ which is nonminimal at the origin there is no a priori way to associate
to $M$ a second order ODE or even a more general PDE system near
the origin, in \cite{nonminimalODE} the Shafikov and the first author
found an injective correspondence between  nonminimal at the
origin and spherical outside the complex locus hypersurfaces
$M\subset\CC{2}$ and certain {\em singular} complex ODEs $\mathcal E(M)$ with an
isolated  singularity at the origin. It is possible to extend this construction to the non-spherical case, which we do in Section 3.



\subsection{Equivalences and symmetries of  ODEs}\label{sub:equiv2ndorder}
We start with a description of the jet prolongation approach to
the equivalence problem (which is a simple interpretation of a
more general approach in the context of {\em jet bundles}). We refer to the excellent sources \cite{olver}, \cite{bluman} for more details and collect the necessary 
prerequisites here.
In what follows all variables are assumed to be complex, all
mappings biholomorphic, and all ODEs to be defined near their zero
solution $y(x)=0$.

Consider two ODEs,  $\mathcal E$ given by $y^{(k)}=\Phi(x,y,y',...,y^{(k-1)} )$
and
$\tilde{\mathcal E}$ given by $ y^{(k)}=\tilde\Phi(x,y,y',...,y^{(k-1)})$, where the functions
$\Phi$ and $\tilde\Phi$ are holomorphic in some neighbourhood of the
origin in $\CC{k+1}$. We say that a germ  of a biholomorphism
$F \colon (\CC{2},0)\lr(\CC{2},0)$  transforms $\mathcal E$ into
$\tilde{\mathcal E}$, if it sends (locally) graphs of solutions of
$\mathcal E$ into graphs of solutions of $\tilde{\mathcal E}$.
We define the {\em $k$-jet space} $J^{(k)}$ to be the $(k+2)$-dimensional
 linear space with coordinates $x,y,y_1,...,y_{k}$,
which correspond to the independent variable $x$, the dependent
variable $y$ and its derivatives up to order $k$, so that we can
naturally consider $\mathcal E$ and $\tilde{\mathcal E}$ as complex
submanifolds of $J^{(k)}$.

For any biholomorphism $F$ as above one may consider its
{\em $k$-jet prolongation} $F^{(k)}$, which is defined on a neighbourhood of the origin in $\CC{k+2}$ as follows.
The first two components of the mapping $F^{(k)}$
coincide with those of $F$. To obtain the remaining components  we
denote the coordinates in the preimage by $(x,y)$ and in the
target domain by $(X,Y)$. Then the derivative $\frac{dY}{dX}$ can
be symbolically recalculated, using the chain rule, in terms of
$x,y,y'$, so that the third coordinate $Y_1$ in the target jet
space becomes a function of $x,y,y_1$. In the same manner one
obtains the remaining components of the prolongation of the
mapping $F$. Thus, for differential equations of order $k$, {\em the mapping $F$ transforms the ODE $\mathcal E$ into $\tilde{\mathcal
E}$ if and only if the prolonged mapping $F^{(k)}$ transforms
$(\mathcal E,0)$ into $(\tilde{\mathcal E},0)$ as submanifolds in the
jet space $J^{(k)}$}. A similar statement can be formulated for
certain singular differential equations, for example, for linear
ODEs (see, e.g., \cite{ilyashenko}).

The jet prolongation approach becomes particularly effective in its application to automorphisms (point symmetries) of differential equations. A vector field $L$ is called a {\em Lie point symmetry} of a differential equation, if its local flow $F_t$ consists of symmetries of the equation. By differentiating the $k$-jet prolongation formulas for $F_t$ with respect to the time $t$, one gets a new vector field $L^{(k)}$ in the $k$-jet space. Then S.~Lie\rq{}s Theorem says that

\smallskip

{\em a vector field $L$ is a Lie symmetry of a differential equation $\mathcal E$ of order $k$ if and only if its $k$-jet prolongation $L^{(k)}$ is tangent to $({\mathcal E},0)$ (considered as a submanifold in the
jet space $J^{(k)}$).}

 \smallskip
 
  S.Lie\rq{}s Theorem is a powerful tool for investigating symmetries of differential equations in view of the fact that the formulas for the prolonged vector fields $L^{(k)}$ are explicit (in terms of the initial vector field $L$) and relatively simple (see \cite{bluman},\cite{olver}). In addition, importantly, {\em $L^{(k)}$ is always polynomial in the jet variables $y_1,...,y_k$}. This property is crucial for the constructions used in the  paper.

\subsection{Complex linear differential equations with an isolated singularity} \mbox{}
Complex linear ODEs are important classical objects, whose geometric
interpretations are plentiful. We
refer to the excellent sources \cite{ilyashenko}, \cite{ai},
\cite{bolibruh}, \cite{vazow},\cite{coddington}  on  complex linear differential equations, gathering here the facts that we will need in
the sequel.

A {\em first order linear system} of $n$ complex ODEs in a domain
$G\subset\CC{}$ (or simply a {\em linear system} in a domain $G$)
is a holomorphic ODE system $\mathcal L$ of the form
$y'(x)=A(x)y(x)$, where $A(x)$ is a
holomorphic function on $G$, taking values in the space of  $n\times n$ matrices,
 and $y(x)=(y_1(x),...,y_n(x))$
is an $n$-tuple of (unknown) functions. The set of  solutions of $\mathcal L$
near a point $p\in G$ form a linear space of dimension~$n$.
Moreover, any germ at $p\in G$ of a solution of $\mathcal L$
extends analytically along any path $\gamma\subset G$, starting
at~$p$, so that any solution $y(x)$ of $\mathcal L$ is defined
globally in $G$ as a (possibly multiple-valued) analytic function.
 A \em fundamental system of solutions for $\mathcal L$ \rm
is a matrix whose columns form some collection of $n$ linearly
independent solutions of $\mathcal L$.

If $G$ is a punctured disc, centered at $0$, we say
that $\mathcal L$  is a system {\em with an isolated singularity at $x=0$}.
An important (and sometimes even a complete) characterization
of an isolated singularity is its {\em  monodromy operator}, which is
defined as follows. If $Y(x)$ is some fundamental system of
solutions of $\mathcal L$ in $G$, and $\gamma$ is a simple loop
about the origin, then it is not difficult to see that the
monodromy of $Y(x)$ with respect to $\gamma$ is given by the right
multiplication by a constant nondegenerate matrix $M$, called the
{\em monodromy matrix}.  The matrix $M$ is defined up to a
similarity, so that it defines a linear operator
$\CC{n}\lr\CC{n}$, which is  called the monodromy operator of the
singularity.

If $A(x)$ has a pole at the isolated singularity $x=0$,
we say that the system has a {\em meromorphic singularity}. As
the solutions of $\mathcal L$ are holomorphic in any proper sector
$S\subset G$ of a sufficiently small radius with vertex at
$x=0$, it is important to study the behaviour of the solutions as
$x\rightarrow 0$. If all solutions of $\mathcal L$ admit a bound
$||y(x)||\leq C|x|^a$  in any such sector (with some constants
$C>0,\ a\in \mathbb R$, depending possibly on the sector), then
$x=0$ is a {\em regular singularity}, otherwise it is
 an {\em irregular singularity}.  In particular, if the monodromy is trivial, then the singularity is regular if and
only if all the solutions of $\mathcal L$ are meromorphic at $x=0$.

L.~Fuchs introduced the following condition: the singular point
$x=0$ is {\em Fuchsian}, if $A(x)$ is meromorphic  at
$x=0$ and has a pole of order $\leq 1$ there. The Fuchsian
condition turns out to be sufficient for the regularity of a
singular point. Another remarkable property of Fuchsian
singularities can be described as follows. We say that
two complex
linear systems with an isolated singularity $\mathcal L_1,\mathcal
L_2$ are {\em (formally) equivalent},  if there exists a (formal)
transformation $F:\,(\CC{n+1},0)\lr(\CC{n+1},0)$ of the form
$F(x,y)=(x,H(x)y)$ for some (formal) invertible matrix-valued
function $H(x)$, which transforms (formally) $\mathcal L_1$ into
$\mathcal L_2$. It turns out that two Fuchsian systems are
formally equivalent if and only if they are holomorphically equivalent
 (in fact, any formal equivalence between them as
above must be convergent). Any Fuchsian system can be brought to a
special polynomial  form (in the sense that the matrix $xA(x)$ is
polynomial) called the {\em Poincare-Dulac normal form for Fuchsian
systems}, and moreover, the normalizing transformation is
always convergent.

However, in the {\em non}-Fuchsian case the behavior of solutions
and mappings between  linear systems is totally different.
Generically, solutions of a non-Fuchsian system
\[y'=\frac{1}{x^m}B(x)y, \quad m\geq 2\]  do {\em not } have polynomial
growth in sectors, and formal equivalences between non-Fuchsian
systems are divergent, as a rule. Also the transformation
bringing a non-resonant non-Fuchsian system to a special
polynomial form called {\em Poincare-Dulac normal form for
non-Fuchsian systems}  is usually also divergent. As some
compensation for this divergence phenomenon, we formulate below a
remarkable result,  {\em Sibuya's sectorial normalization
theorem},   which is of fundamental importance for our
constructions.

For a system $y'=\frac{1}{x^m}B(x)y,\,m\geq 2$
which is non-resonant (i.e., the leading matrix $B_0=B(0)$ has
pairwise distinct eigenvalues $\{\lambda_1,...,\lambda_n\}$) we
call each of the $2(m-1)$ rays
$R_{ij}=\left\{\re\left((\lambda_i-\lambda_j)x^{1-m}\right)=0\right\},\,i,j=1,...,n,\,i\neq
j,$  a {\em separating ray for the system}. 

\begin{theorem} [Y.~Sibuya, 1962, see
\cite{sibuya},\cite{ilyashenko}]\label{thm:sibuya} Assume that  a non-Fuchsian
linear system $\mathcal
E$
\[y'=\frac{1}{x^m}B(x)y,\quad m\geq 2\]
is non-resonant
 and $S \subset (\CC{}, 0)$
is an arbitrary sector with vertex at $0$ not containing two
separating rays for any pair of the eigenvalues. Then for any
formal conjugacy $x\mapsto x,\,y\mapsto \hat H(x)y$, conjugating
the system with its Poincare-Dulac polynomial normal form, there
exists a holomorphic  function
$H_S(x)$ defined in $S$ and taking values in $\mbox{GL}(n,\CC{})$ such
that
$H_S(x)$ asymptotically represents $\hat H(x)$
in $S$ (see Introduction for this concept), and $x\mapsto x,\,y\mapsto H_S(x)y$ conjugates $\mathcal E$
with its Poincare-Dulac normal form in $S$  for . If a sector $S$ has
opening bigger than $\frac{\pi}{m-1}$, then the sectorial
normalization $H_S(x)$ is unique.
\end{theorem}

\noindent Alternatively, one can require for the uniqueness in Sibuya\rq{}s theorem that the sector $S$ contains a separating ray for each pair of eigenvalues of the leading matrix. 

Importantly, one can use Sibuya\rq{}s theorem for establishing asymptotic expansion properties of solutions, thanks to the fact for a system in the Poincare-Dulac normal form the solutions can be described explicitly (see \cite{vazow},\cite{coddington} for details). 

We note
that the holomorphic sectorial normalization in \autoref{thm:sibuya} does usually
{\em not} extend to one holomorphic near the origin. The reason is that, somewhat surprisingly, the sectorial
normalization $H_S(x)$ might change from sector to sector. The corresponding
solutions transported from the solutions of the system in normal form 
by $H_S(x)$ differ from each other by means
of multiplication by a constant matrix $C\in \mbox{GL}(n,\CC{})$
called a  {\em Stokes matrix}.  This phenomenon is known as
the {\em Stokes phenomenon},  and the entire collection $\{C_{ij}\}$ of
Stokes matrixes, corresponding to all separating rays, is called
the {\em Stokes collection}. For a generic equation the collection of 
Stokes matrixes contains nontrivial (non-identity) matrices. 

A scalar linear complex ODE of order $n$ in a
 domain $G\subset\CC{}$ is an ODE $\mathcal E$ of the form
$$z^{(n)}=a_{n}(x)z+a_{n-1}(x)z'+...+a_1(x)z^{(n-1)},$$ where $\{a_1(x), \dots a_n (x)\}$ is
a given collection of holomorphic functions in $G$ and $z(x)$ is
the unknown function. By a reduction of $\mathcal E$ to a first
order linear system (see the above references and
also~\cite{vyugin} for various approaches of doing that) one can
naturally transfer most of the definitions and facts, relevant to
linear systems, to scalar equations of order $n$. The main
difference here is contained in the appropriate definition of
Fuchsian: a singular point $x=0$ for an ODE $\mathcal E$ is said to be
{\em Fuchsian},  if the orders of poles $p_j$ of the functions
$a_j(x)$ satisfy the inequalities $p_j\leq j$, $j=1,2,\dots,n$. It
turns out that the condition of Fuchs becomes also necessary for
the regularity of a singular point in the case of $n$-th order
scalar ODEs.

Further information on the classification of isolated
singularities (including details of Poincare-Dulac normalizations
in the Fuchsian and non-Fuchsian cases respectively) can be found
in \cite{ilyashenko}, \cite{vazow} or \cite{coddington}.

\section{Analytic continuation of infinitesimal automorphisms}\label{sec:cont}

\subsection{Complete system at $m$-admissible points}

We use the notations from the Introduction and for a Levi-nonflat real-analytic hypersurface $M\subset\CC{2}$, containing a complex hypersurface $X$,  consider the decomposition of its Levi degeneracy set $\Sigma$ as $\Sigma=X\cup \Sigma_1$. If $p\notin X\cap\Sigma_1$, then, according to \cite{nonminimalODE}, there exists local holomorphic coordinates at $p$ such that $p$ is the origin, the complex locus $X=\{w=0\}$, and $M$ is represented by an equation 
\begin{equation}\label{madmissiblereal}
v=u^m\left(\pm |z|^2+\sum_{k,l\geq 2}h_{kl}(u)z^k\bar z^l\right).
\end{equation}
Alternatively, one can consider the complex defining equation at the origin which looks as
\begin{equation}\label{madmissible}
w=\bar w e^{\pm i\bar w^{m-1}\varphi(z,\bar z,\bar w)},\quad\mbox{where}\quad\varphi(z,\bar z,\bar w)=z\bar z+\sum_{k,l\geq 2}\varphi_{kl}(\bar w)z^k\bar z^l.
\end{equation}
Here $m\geq 1$ is the nonminimality order at $p$. 
A real hypersurface defined by an equation  as in \eqref{madmissible} is called 
{\em $m$-admissible} (see \cite{divergence,nonanalytic}). 
Importantly, an $m$-admissible hyperurface automatically 
contains the curve $\Gamma$ defined in  \eqref{Gamma}. Depending on the sign in \eqref{madmissible},
 $M$ is called {\em positive} or {\em negative} respectively (note that this 
 notion might depend on the choice of coordinates). Let us denote by $\tilde U\subset \CC{2}$ a 
 neighborhood of the origin, containing no points from $\Sigma_1$, where the Segre family 
 of $M$ is defined. We claim that there exists a unique holomorphic near the origin in
  $\CC{3}$ function $\Phi(z,w,\zeta)$ of the form $\Phi=O(w^m\zeta^2),$ and a polydisc
   $U\subset \tilde U$, centered at the origin, such that all Segre varieties $Q_p$ of
    $M$ with $p\in U,\,p\notin X$ (considered as graphes $w=w_p(z)$) satisfy the
     second order {\em singular} ODE 
\begin{equation}\label{ODE}
w\rq{}\rq{}=\Phi\left(z,w,\frac{w\rq{}}{w^m}\right).
\end{equation}
The latter fact was proved in \cite{nonminimalODE} for the special case of hypersurfaces, spherical at a generic point. However, it is not difficult to extend the proof for the case of a general $m$-admissible hypersurface, which we do now.

 Assume that $M$ is positive first.  Take then an arbitrary $p=(\xi,\eta)\in
U\setminus X$ and consider Segre
varieties $Q_p$ of $M$ as graphs 
\begin{equation}\label{segre}
w=w_p(z)=\bar\eta
e^{i\bar\eta^{m-1}\varphi(z,\bar \xi,\bar\eta)},\quad \varphi(z,\bar \xi,\bar\eta)=z\bar\xi+O(z^2\bar\xi^2).
\end{equation} 
We have
\begin{equation}\label{approximation}
w=\bar\eta+O(z\bar\xi\bar\eta^m),\quad \frac{w'}{w^m}=-i\bar\xi+O(z\bar\xi),
\quad w''=O(\bar\xi^2\bar\eta^m).
\end{equation}
and use the relations \eqref{approximation} in order to obtain a
second order ODE satisfied by all $Q_p,\,p\in
U\setminus X$. An application of  the
implicit function theorem to the first two equations in   \eqref{approximation}
yields  functions
$\Lambda(z,w,\zeta)=i\zeta+O(z\zeta)$ and $\Omega(z,w,\zeta)=w+O(zw\zeta)$, such that
$$\bar\xi=\Lambda\left(z,w,\frac{w'}{w^m}\right),\,\bar\eta=\Omega\left(z,w,\frac{w'}{w^m}\right).$$
 Substituting
$\bar\xi=\Lambda(z,w,\frac{w'}{w^m}),\,\bar\eta=\Omega(z,w,\frac{w'}{w^m})$
into the equation for $w''$ in \eqref{approximation} gives us a
second order ODE
\eqref{ODE} for some
function $\Phi(z,w,\zeta)$, holomorphic in a polydisc $ W \subset\CC{3}$, centered at the origin (compare this with the
elimination procedure in Section~2). The ODE \eqref{ODE} is
satisfied by all $Q_p$ with $p\in
U\setminus X$. In addition, the function
$\Phi(z,w,\zeta)$ also satisfies $\Phi(z,w,\zeta)=O(\zeta^2w^m)$. 

The proof in the case when $M$ is negative is analogues. 
The ODE \eqref{ODE} is called {\em the associated ODE for $M$} and is denoted by $\mathcal E(M)$. 

Let now $L=P\dz+Q\dw$ be a smooth infinitesimal CR automorphism of $M$, defined on a 
neighbourhood $V\subset U$ of $q$, where $q\in (M\setminus X)\cap U$. We recall
to the reader that $L$ {\em extends holomorphically to a neighbourhood of $q$} (because $q$ is a 
minimal point); we will use this fact without further mentioning in the sequel. 
Denote by $\Delta\subset \CC{}$ a disc such 
that the right-hand side in \eqref{ODE} is well defined in the domain
 $\tilde V:=V\times \Delta\subset\CC{2,1}$. Since the real flow of $L$ near $q$ preserves 
 the Segre family of $M$,  
  $L$ is also a Lie symmetry of the {\em nonsingular} in $\tilde V$ ODE \eqref{ODE}. 
  Thus the second prolongation $L^{(2)}$ of $L$ is tangent to \eqref{ODE}
   over this domain. We have 
\begin{equation}\label{prolong}
L^{(2)}=P(z,w)\dz+Q(z,w)\dw+Q^{(1)}(z,w,w_1)\frac{\partial}{\partial w_1}+Q^{(2)}(z,w,w_1,w_2)\frac{\partial}{\partial w_2},
\end{equation}
where 
\begin{gather*}w_1:=w\rq{},\quad w_2:=w\rq{}\rq{},\\
 Q^{(1)}=Q_z+\bigl(Q_w-P_z\bigr)w_1-P_w(w_1)^2,\\
Q^{(2)}=Q_{zz}+\bigl(2Q_{zw}-P_{zz}\bigr)w_1+\bigl(Q_{ww}-2P_{zw}\bigr)(w_1)^2-\\
-P_{ww}(w_1)^3
+\bigl(Q_w-2P_z)w_2-3P_w w_1w_2 
\end{gather*}
(see, e.g., \cite{bluman}). Restricting onto $w_2=\Phi\bigl(z,w,\frac{w_1}{w^m}\bigr)$, we get
\begin{equation}\label{tangency}
Q^{(2)}\left|_{w_2=\Phi\bigl(z,w,\frac{w_1}{w^m}\bigr)}\right.=\Phi_z\left(z,w,\frac{w_1}{w^m}\right)\cdot P + \Phi_w\left(z,w,\frac{w_1}{w^m}\right)\cdot Q+\frac{1}{w^m}\Phi_{\zeta}\left(z,w,\frac{w_1}{w^m}\right)\cdot Q^{(1)}
\end{equation}
for all $(z,w,w_1)\in\tilde V$. Note that in view of the fact that \eqref{prolong} is polynomial in $w_1$, for any fixed $(z,w)\in V$ both sides in \eqref{tangency} are holomorphic in $w_1$ at the origin. 
Let us expand $\Phi$ as 
$$\Phi=w^m\bigl(\Phi_2(z,w)\zeta^2+\Phi_3(z,w)\zeta^3+\Phi_4(z,w)\zeta^4+O(\zeta^5)\bigr),$$
where $\Phi_2,\Phi_3,\Phi_4$ are holomorphic at the origin, and denote
$$a:=\frac{1}{w^{m}}\Phi_2,\quad b:=\frac{1}{w^{2m}}\Phi_3,\quad c:=\frac{1}{w^{3m}}\Phi_4.$$
Then, gathering in \eqref{tangency} for any fixed $(z,w)\in V$ terms with $(w_1)^0, (w_1)^1, (w_1)^2, (w_1)^3$, we get respectively
\begin{equation}\label{initial}
\begin{aligned}
Q_{zz} &=0,\\
 2Q_{zw}-P_{zz}&=2aQ_z,\\
 Q_{ww}-2P_{zw}&=a(-Q_w+2P_z)+a_zP+a_wQ+3bQ_z+2a(Q_w-P_z),\\
 P_{ww}&=b(Q_w-2P_z)-aP_w-b_zP-b_wQ-4cQ_z+3b(P_z-Q_w).
\end{aligned}
\end{equation}

We then differentiate each of the four identities \eqref{initial} in $z$ and $w$ once. We get a system $S$ of $8$ equations of order $3$, expressing the {\em complete} set of third order derivatives of $P$ and $Q$ as linear combinations of lower order derivatives with coefficients holomorphic in $z$ and meromorphic in $w$ at the origin. (The order of the pole of $w$ at $0$ here does not exceed $3m+1$). Introduce
$$y(z,w):\quad\CC{2}\lr \CC{12},\quad y:=(P,Q,P_z,P_w,Q_z,Q_w,P_{zz},P_{zw},P_{ww},Q_{zz},Q_{zw},Q_{ww}).$$
Then, taking into account the relations between the components of the vector function $y$ and the system $S$, we obtain the following system of first order linear PDEs:
\begin{equation}\label{complete}
\frac{\partial y}{\partial z}=A(z,w)y,\quad \frac{\partial y}{\partial w}=B(z,w)y,
\end{equation}
where $A(z,w),B(z,w)$ are holomorphic in $U\setminus X$  $12\times 12$ matrix functions with a meromorphic singularity of order $\leq 3m+1$ along $X=\{w=0\}$. 

\subsection{Extension along a path} We now need the following
\begin{lemma}\label{polydisc}
Let $\Delta_1,\Delta_2,\Delta\rq{}_1,\Delta\rq{}_2$ be discs in $\CC{}$, $\Delta\rq{}_1\subset\Delta_1,\,\Delta\rq{}_2\subset\Delta_2,\quad \Delta_1\times\Delta_2\subset U\setminus X$. Then any solution $y(z,w)$ of the system \eqref{complete} in the polydisc $\Delta\rq{}_1\times\Delta\rq{}_2$ extends holomorphically to $\Delta_1\times\Delta_2$.
\end{lemma}
\begin{proof}
Consider first the case when both polydiscs have the same center. Without loss of generality, assume then that they are both centered at $0$. We have 
\begin{equation}\label{expansion}
y(z,w)=\sum_{j=1}^\infty \varphi_j(w)z^j,\quad (z,w)\in\Delta\rq{}_1\times\Delta\rq{}_2,\quad \varphi_j(w)\in\mathcal O(\Delta\rq{}_2).
\end{equation}
Consider now the first equation in \eqref{complete} for a fixed $w=w_0\in\Delta\rq{}_2$. Then we get a first order linear ODE for $y(z,w_0)$ in $\Delta\rq{}_1$. As a well known fact (e.g., \cite{ilyashenko}), any local solution of it extends to $\Delta_1$ holomorphically. Thus the series \eqref{expansion}, considered for a fixed $w=w_0$, converges in $\Delta_1$. Hence the formula \eqref{expansion} provides a holomorphic extension of $y$ to $\Delta_1\times \Delta\rq{}_2$. Similar arguments then give the holomorphic extension of $y(z,w)$ from $\Delta_1\times \Delta\rq{}_2$ to $\Delta_1\times \Delta_2$, as required.

In the general case, we use finitely many iterations of the above argument to obtain an extension of $y(z,w)$ to a polydisc $\Delta_1\rq{}\rq{}\times\Delta_2\rq{}\rq{}$ having the same center as $\Delta_1\times\Delta_2$. Applying then the above argument, we obtain the desired extension to $\Delta_1\times\Delta_2$ (note that, since a polydisc is simply-connected, this extension must agree with $y(z,w)$ in the polydisc  $\Delta\rq{}_1\times\Delta\rq{}_2$).
\end{proof}

\begin{proof}[Proof of \autoref{extension}] Since $q$ is a finite type point, $L$ extends analytically to an open neighborhood of $q$ in $\CC{2}$. Now application of \autoref{polydisc} to a finite sequence of polydiscs, contained in $U\setminus X$ and covering a path $\gamma\subset U\setminus X$ starting at $q$, proves statement (i) in \autoref{extension}. To prove statement (ii), we apply the Blow-up Lemma (see Section 2) and then statement (i). 
\end{proof}

 \subsection{The infinitesimal monodromy} We use the (multi-valued) extension phenomenon, discovered in the previous section, to introduce the notion of an {\em infinitesimal monodromy} of a nonminimal hypersurface. We consider here the more transparent case when $M\setminus X$ is Levi-nondegenerate, even though in the general setting this can also be done.

  Let $M$ be nonminimal at a point $p$, and $X\ni p$ its complex locus. Assume that $(M\setminus X)\cap U$ is Levi-nondegenerate and fix a point $q\in (M\setminus X)\cap U$, where $U$ is a neighborhood where the assertion of \autoref{extension} holds. Consider a vector field $L=P\dz+Q\dw\in\hol{}(M,q)$ and a loop $\gamma$, starting at $q$ and generating $\pi_1(U\setminus X)$. It turns out that the analytic continuation of $L$ along $\gamma$  (possible by \autoref{extension}) is again an infinitesimal automorphism of $M$: Indeed, if $M$ is given in $U$ by a complex defining equation $w=\rho(z,\bar z, \bar w)$, then the tangency condition for $L$ becomes
$$Q(z,w)=\rho_z(z,\bar z,\bar w)P(z,w)+\rho_{\bar z}(z,\bar z,\bar w)\bar P(\bar z,\bar w)+\rho_{\bar w}(z,\bar z,\bar w)\bar Q(\bar z,\bar w),\quad \text{ if } w = \rho(z,\bar z, \bar w).$$ 
The latter condition is invariant under the analytic continuation of $P$ and $Q$, so that the continuation of $L$ is again an infinitesimal automorhism and hence satisfies \eqref{systeminw}. We obtain a correspondence
\begin{equation}\label{psi}
\psi_\gamma:\quad \hol{}(M,q)\mapsto \hol{}(M,q).
\end{equation}
Clearly, $\psi_\gamma$ is linear, injective and invariant under the Lie bracket operation, so that $\psi$ is {\em an automorphism} of the finite-dimensional Lie algebra $\mathfrak g=\hol{}(M,q)$. The automorphism $\psi_\gamma$ is independent of the choice of the basic point $q$ and the loop $\gamma$. 

\begin{definition}\label{infmonodromy}
We call $\psi_\gamma$ {\em the infinitesimal monodromy} of $M$. 
\end{definition}

The infinitesimal monodromy is {\em a biholomorphic invariant} of a nonminimal hypersurface. This invariant is additional to the Lie algebra structure of $\hol{}(M,p)$ and to that of the infinitesimal automorphism algebra at Levi-nondegenerate points $q$. The infinitesimal monodromy should be compared with the monodromy of a map into a quadric, introduce in \cite{nonminimal} for the case when $M\setminus X$ is Levi-nondegenerate and spherical.

\subsection{Sectorial extendability of CR-objects}
In this section we prove

\begin{proposition}\label{correctness} Let $M\subset\CC{2}$ be a real-analytic nonminimal at the origin Levi-nonflat hypersurface. 

\smallskip

\noindent (i) If $D$ is a tangential sectorial domain for  $M$ at the origin, then, for any local holomorphic coordinates where the complex locus of $M$ is given by \eqref{locus} and $M$ contains the curve \eqref{Gamma}, $D$ contains a standard tangential sectorial domain \eqref{standard}. 

\smallskip
 
\noindent (ii) If a $C^\infty$ CR-function $f$ on $M$  defined near the origin is sectorially extendable to some tangential sectorial domain $D$, then it is also sectorially extendable to a tangential sectorial domain $\tilde D$ corresponding to another choice of coordinates.

\end{proposition}   

In other words, the definitions of a tangential sectorial domain and the sectorial extendability are independent of the choice of coordinates where $M$ contains the curves \eqref{locus},\eqref{Gamma}. 

\begin{proof}
We first prove (i). Let us denote by $(z,w)$ the given local holomorphic coordinates,  by $(z^*,w^*)$ some coordinates where $D$ has the standard form \eqref{standard}, and by $F:\,(z^*,w^*) \mapsto (z,w)$ the biholomorphism between them. Then $F=G\circ H$, where $H$ has the form $z= z^*+p(w^*),\,w= q(w^*),\,q(w)\in\RR{}\{w\},$ and $G$ is a map with the identity linear part of the  form $z=O(z^*),\,w=O(w^*)$. For both $H$ and $G$ it is immediate  that the image of a standard sectorial domain \eqref{standard} contains a standard sectorial domain (in view of $q(w)\in\RR{}\{w\},\,q'(0)\neq 0$). This proves (i). Note that the radii of the disc and the sector in \eqref{standard} can arbitrarily shrink, while the angle of the sectors can be kept bounded from below.  

To prove (ii), we use the same notations (i.e., $\tilde D$ is standard in the coordinates $(z,w)$, while $D$ is standard in the coordinates $(z^*,w^*)$), and the same decomposition  $F=G\circ H$. We use (i) and conclude that, after shrinking $\tilde D$, we have $\tilde D\subset D$.  We next note that 

 \smallskip 

{\it a holomorphic function $f$ defined in a standard sectorial domain admits there an asymptotic representation if and only if  all its partial derivatives $f_{z^kw^l},\,k,l\geq 0$ have a limit in the domain as $(z,w)$ tends to the origin remaining in the domain.}

\smallskip

In view of the chain rule, the latter property for derivatives of order $\leq N$ for the function $f(z^*,w^*)$ implies the same property for derivatives of order $N$ for the function $f(z,w)$ in $\tilde D$, and this proves (ii) and the proposition.

\end{proof}

\subsection{Proof of the sectorial extension result} Let $L=P\dz+Q\dw$ be a $C^\infty$ infinitesimal CR automorphism
 of a hypersurface \eqref{madmissible} at the origin, defined in $M\cap U$ for a sufficiently small 
 neighborhood $U$ of the origin in $\CC{2}$. We consider first the case when $(M\setminus X)\cap U$ is
  Levi-nondegenerate.  Then for any $q\in (M\setminus X)\cap U$ the equations \eqref{initial} are applicable. The first equation in \eqref{initial} yields 
  \begin{equation}\label{findQ}
  Q(z,w)=Q_0(w)+Q_1(w)z,
  \end{equation}
  for $(z,w)$ in a sufficiently small neighbourhood of $q$. 
The fact that $Q$ is a $C^\infty$ CR-function on $M$ implies that  
the restrictions of $Q_0,Q_1$ onto the real line are $C^\infty$ functions at the origin, which extend 
holomorphically to a neighborhood of the set $w\in(-\varepsilon,0)\cup(0,\varepsilon)$ in $\CC{}$ for some $\varepsilon>0$.   Applying then the second equation in \eqref{initial}, we obtain
\begin{equation}\label{findP}
P(z,w)=P_0(w)+P_1(w)z+Q_1\rq{}(w)z^2-2Q_1(w)\tilde a(z,w),
\end{equation}
where $\tilde a=O(z^2),\,\, \tilde a_{zz}=a$, and $P_0,P_1$ are described similarly to $Q_0,Q_1$. Let us introduce the vector function
$$u(w):=\bigl(P_0,P_1,P_0\rq{},P_1\rq{},Q_0,Q_1,Q_0\rq{},Q_1\rq{}\bigr).$$
Then, substituting the above obtained formulas for $P,Q$ into the third and the fourth equations in \eqref{initial}, gathering there terms of degree $0$ and $1$ in $z$, and using the relations between the components of the vector $u(w)$, we get a system
\begin{equation}\label{systeminw}
\frac{du}{dw}=C(w)u,
\end{equation}
where $C(w)$ is a meromorphic at the origin $8\times 8$ matrix function (with the order of pole not exceeding $3m$).
We now need 

\begin{proposition}\label{solutions}
Let $u(w)$ be a solution of a linear ODE system 
\begin{equation}\label{linearsystem}
y\rq{}=A(w)y,\quad w\in\CC{},\quad y\in\CC{n}
\end{equation}
with an isolated meromorphic singularity at the point $w=0$, which is well defined and is $C^\infty$ smooth on the interval $w\in[0,\varepsilon)$. Then there exist a sector $S^+$, as in \eqref{sectors}, such that $u$ extends holomorphically to $S^+$ and admits there an asymptotic expansion, coincident with the Taylor expansion of $u(w)$ at the origin.
\end{proposition} 

\begin{proof} Consider first the case when the system \eqref{linearsystem} is non-resonant. Let $s\geq 2$ denotes the order of the singularity in \eqref{linearsystem} and $\lambda_1,...,\lambda_n$ the eigenvalues of the leading matrix (if $s=1$, the singularity is Fuchsian and $u$ in fact extends to $w=0$ holomorphically). Then (e.g. \cite{vazow})
there exists a sufficiently small sector $S^+$, as in \eqref{sectors}, such that the fundamental system of solutions of \eqref{linearsystem} in $S^+$ looks as 
$$F(w)w^G e^{T(w)}.$$
Here the matrix function $F(w)=\bigl( f_{ij}(w)\bigr)_{i,j=1}^n$ is holomorphic in $S^+$ and admits there an asymptotic expansion 
\[F \sim E+C_1w+..., \]
 where $E$ denotes the identity matrix, $G$ is a constant diagonal matrix, and
 \[T(w)=w^{1-s}(T_0+T_1w+...+T_{s-2}w^{s-2}),\] where all $T_l$ are diagonal and $T_0=\frac{1}{1-s}\mbox{diag}\{\lambda_1,...,\lambda_n\}$. Hence, if $F^j$ denote the columns of $F$ and $t_j(1/w)$ the diagonal elements of $T(w)$, for some constants $c_1,...,c_n$ we have
\begin{equation}\label{columns}
u(w)= \begin{pmatrix}
	u_1 (w) \\ \vdots \\ u_n(w)
\end{pmatrix}=\sum_{j=1}^n c_j F^j(w)w^{\alpha_j}e^{t_j(1/w)}.
\end{equation}
Let us for the sequel of the proof call an integer $j\in[1,n]$ {\em good} if either $c_j=0$, or the exponent $e^{t_{j}(1/w)}$ is flat, or $t_j(w)\equiv 0$, and call it {\em bad} otherwise. We aim to show that there are no bad integers.  We do so in two steps.

\smallskip

\noindent{\bf Step 1.} We claim that for all bad inregers we have $\re t_j(1/w)|_{w\in \RR{}}\equiv0$. Indeed, arguing by contradiction we choose  an integer $j_0$ such that $c_{j_0}\neq 0$ and  $w^{\alpha_{j_0}}e^{t_{j_0}(1/w)}$ has the maximal growth on $\RR{+}$  as $w$ tends to $0$ 
(in the sense that the ratio with all other analogues terms with $c_j\neq 0$ is bounded on the real line as $w$ tends to $0$; note that $j_0$ need not be unique). Note that, by the contradiction assumption, the inverse of $w^{\alpha_{j_0}}e^{t_{j_0}(1/w)}$ is flat. Without loss of generality assume $j_0=1$. Recall that $f_{11}(w)\rightarrow 1,\,f_{1j}(w)\rightarrow 0,\,j\neq 1$ as $w\rightarrow 0$, and $u_1(w)$ has an asymptotic expansion on $\RR{+}$  as $w$ tends to $0$. From \eqref{columns} we have
\begin{equation}\label{u1}
u_1(w)=\sum_{c_j\neq 0}c_jf_{1j}(w)w^{\alpha_{j}}e^{t_{j}(1/w)}.
\end{equation}
Dividing \eqref{u1} by $w^{\alpha_{1}}e^{t_{1}(1/w)}$, we get a contradiction, proving our claim.

\smallskip

\noindent{\bf Step 2.}  Recall that each $t_j(1/w)$ is a polynomial in $\frac{1}{w}$ of degree at most $s-1$. Amont all bad inregers $j$ (assuming that they exist), consider the maximal possible degree of such polynomial $t_j$. Without loss of generality let us assume that this maximal degree is $s-1$. Next, among all bad integers with maximal degree of $t_j$ being equal to $s-1$, consider the one with the maximal growth of $w^{\alpha_j}$ (in the above sense). Without loss of generality we assume that the latter integer is $1$. Recall also that (according to Step 1) for all bad integers we have $\left|e^{t_{j}(1/w)}\right|\equiv 1$ for $w\in\RR{}$. We then differentiate \eqref{u1} sufficiently many times. Then it follows from the construction that after $k$ differentiations with $k$ large enough the first term in the right-hand side of \eqref{u1} has growth strictly exceeding the growth of the other terms, in the sense that the right-hand side has the form 
\begin{equation}\label{differed}
c_1f_{11}(w)w^{\alpha_{1}}\left(-\frac{1}{w^2}\right)^k\left(t_1'\left(\frac{1}{w}\right)\right)^ke^{t_{1}(w)}\bigl(1+o(1)\bigl)
\end{equation}
(we again used the fact that $f_{11}(w)\rightarrow 1,\,f_{1j}(w)\rightarrow 0,\,j\neq 1$ as $w\rightarrow 0$). However, the left hand side has an asymptotic expansion in the positive real line and, in particular, has a limit as $w\rightarrow 0$, while for $k$ large enough \eqref{differed} does not have a limit as $w\rightarrow 0$, which gives a contradiction and proves that bad integers do not exist, as required.

Thus the sum in the right-hand side of \eqref{columns} consists of terms which  either are flat or have the form $c_jF^j(w)w^{\alpha_j},\,c_j\neq 0$.  Picking among the non-flat terms with $\alpha_j\not\in\mathbb{Z}$ the one with the maximal growth of $ w^{\alpha_j}$, we  argue similarly to Step 2 and use differentiations of \eqref{u1} to show that  all such $\alpha_j$ are integers. Now it is already easy to see that all terms in the right-hand side of \eqref{columns} either are flat or have an asymptotic expansion in the sector $S^+$. Note finally that if a function $w^{\alpha_{j}}e^{t_{j}(1/w)}$ is flat on $\RR{+}$ as $w$ tends to $0$, it is also flat in a sufficiently small sector $\tilde S^+$, as in \eqref{sectors}. This immediately implies that $u(w)$ has an asymptotic expansion in a sufficiently small sector.

In the resonant case, according to \cite{coddington} (see Theorem 5.1 there) or \cite{vazow} (see Theorem 19.1 there), for an appropriate sector $S^+$, as in \eqref{sectors}, the fundamental system of solutions of the system \eqref{systeminw} has the form
$$F(w)w^G e^{T(w)}.$$
Here $F(w)$ is holomorphic in $S^+$ and admits there an asymptotic expansion in powers of $w^{1/k}$ for some  integer $k\geq 1$, $G$ is a constant matrix, and $T(w)$ is a polynomial in $w^{-1/k}$ diagonal matrix function. 
Moreover, the formal determinant of $F(w)$ is nonvanishing (i.e., it has a non-zero constant term as a formal power series in $w^{1/k}$). 
We then set $\zeta:=w^{1/k}$ and conslude that
$$u(\zeta^k)=\tilde F(\zeta)\zeta^{kG} e^{\tilde T(\zeta)},$$
where $\tilde F:=F(\zeta^k)$ has an asymptotic expansion in $\zeta$ and $\tilde T:=T(\zeta^k)$ is polynomial in $1/\zeta$. Taking into acount the above nondegeneracy property of $F$, we repeat the argument in the non-resonant case (with minor modifications) to prove that $u(\zeta^k)$ has an asymptotic expansion in powers of $\zeta$ in a sector $\zeta\in\tilde S$. Comparing the latter expansion with the initial Taylor expansion of $u(w)$ in the positive real line, it easy  to conslude that the expansion of $u(\zeta^k)$ in $\tilde S$ has non-zero terms only of the form $c\zeta^{kl},\,l\in\mathbb{Z}_{\geq 0}$. This already implies the desired asymptotic representation of $u(w)$ in the sector  $\{w^{1/k}\in \tilde S\}$, as required.  
\end{proof}

Note that the ray $\RR{+}$ can be replaced in the assertion of \autoref{solutions} by any other ray $d$ with the vertex at $w=0$.

\begin{proof}[Proof of \autoref{sectorial}] The proof is now immediately obtained by applying \autoref{solutions} to the system \eqref{systeminw} and the rays $\RR{\pm}$.

\end{proof}



\section{Fuchsian type hypersurfaces}
 In this section we aim to prove \autoref{analyticity} and \autoref{convergence}. 
 
\subsection{The Fuchsian condition for real hypersurfaces} We first need to investigate the relation between the Fuchsian condition (see the Introduction) and the character of the associated ODE \eqref{ODE}. We start with
 
\begin{proposition}\label{fromreality}
 Let   $M\subset\CC{2}$ be a Fuchsian type hypersurface, given near the origin by an equation of the form \eqref{madmissible}. Then $M$ satisfies the more specified conditions  
 \begin{equation}\label{fuchsianity1}
\begin{aligned}
\mbox{ord}_0\,\varphi_{22}(u)&\geq m-1,\quad \mbox{ord}_0\, \varphi_{23}(u)=\mbox{ord}_0\, \varphi_{32}(u)\geq 2m-2,\quad \mbox{ord}_0\, \varphi_{33}(u)\geq 2m-2,\\ 
\mbox{ord}_0\,
\varphi_{24}(u)&=\mbox{ord}_0\, \varphi_{42}(u)\geq 3m-3, \quad \mbox{ord}_0\, \varphi_{34}(u)=\mbox{ord}_0\, \varphi_{43}(u)\geq 3m-3.
\end{aligned}
\end{equation}
For the real defining function, as in \eqref{madmissiblereal}, the Fuchsianity reads as
\begin{equation}\label{fuchsianityreal}
\begin{aligned}
\mbox{ord}_0\,h_{22}(u)&\geq m-1,\quad \mbox{ord}_0\, h_{23}(u)=\mbox{ord}_0\, h_{32}(u)\geq 2m-2,\quad \mbox{ord}_0\, h_{33}(u)\geq 2m-2,\\ 
\mbox{ord}_0\,
h_{24}(u)&=\mbox{ord}_0\, h_{42}(u)\geq 3m-3, \quad \mbox{ord}_0\, h_{34}(u)=\mbox{ord}_0\, h_{43}(u)\geq 3m-3.
\end{aligned}
\end{equation}
\end{proposition}

\begin{proof}
The assertion of the proposition follows easily from the reality condition \eqref{reality}, the transfer procedure from the real to the complex defining equations, and the Fuchsianity conditions \eqref{fuchsianity}.
\end{proof}
   
\begin{proposition}
Let   $M\subset\CC{2}$ be a Fuchsian type hypersurface, given near the origin by an equation of the form \eqref{madmissible}. Then  the functions $a(z,w),b(z,w),c(z,w)$, as in \eqref{initial}, satisfy
\begin{equation}\label{fuchsianityODE}
\begin{aligned}
\mbox{ord}\,a(0,w)&\geq -1,\quad \mbox{ord}\,a_z(0,w)\geq -1,\quad \mbox{ord}\,a_{zz}(0,w)\geq -1,\\\mbox{ord}\,b(0,w)&\geq -2,\quad \mbox{ord}\,b_z(0,w)\geq -2,\quad \mbox{ord}\,b_{zz}(0,w)\geq -2,\\
\quad  \mbox{ord}\,c(0,w)&\geq -3,\quad \mbox{ord}\,c_z(0,w)\geq -3.
\end{aligned}
\end{equation}
Here for a meromorphic  at the origin function $f(w)$ we denote by $\mbox{ord}_w f$ its order.
\end{proposition}
\begin{proof}
Without loss of generality we assume $M$ to be positive.
Let us expand the right hand side in \eqref{ODE} as

\begin{equation}\label{expandPhi}
\begin{aligned}
\Phi\left(z,w,\frac{w\rq{}}{w^m}\right)&=\frac{1}{w^m}\Bigl(a_0(w)+a_1(w)z+a_2(w)z^2\Bigr)(w\rq{})^2+ 
\frac{1}{w^{2m}}\Bigl(b_0(w)+b_1(w)z+b_2(w)z^2\Bigr)(w\rq{})^3+\\
&+\frac{1}{w^{3m}}\Bigl(c_0(w)+c_1(w)z\Bigr)(w\rq{})^4+O\Bigl(z^3(w\rq{})^2\Bigr)+O\Bigl(z^2(w\rq{})^4\Bigr)+O\Bigl((w\rq{})^5\Bigr)
\end{aligned}
\end{equation}
(note that, as follows from \eqref{ODE}, all $a_j,b_j,c_j$ are {\em holomorphic} at the origin).
Since the Segre family of $M$ looks as \eqref{segre}, we compute 
$$w\rq{}=i\bar\eta^m e^{i\bar\eta^{m-1}\varphi}\varphi_z,\quad
w\rq{}\rq{}=i\bar\eta^m e^{i\bar\eta^{m-1}\varphi}\varphi_{zz}-
\bar\eta^{2m-1} e^{i\bar\eta^{m-1}\varphi}(\varphi_z)^2.$$ 
Hence, plugging into \eqref{ODE} and taking \eqref{expandPhi},\eqref{approximation} into account, we obtain the identity
\begin{multline}\label{huge}
\varphi_{zz}(z,\bar\xi,\bar\eta)=-i\bar\eta^{m-1}(\varphi_z)^2+i\Bigl(a_0(w)+a_1(w)z+a_2(w)z^2\Bigr)e^{i(1-m)\bar\eta^{m-1}\varphi}(\varphi_z)^2-\\
-\Bigl(b_0(w)+b_1(w)z+b_2(w)z^2\Bigr)e^{i(2-2m)\bar\eta^{m-1}\varphi}(\varphi_z)^3
-i\Bigl(c_0(w)+c_1(w)z\Bigr)e^{i(3-3m)\bar\eta^{m-1}\varphi}(\varphi_z)^4+\\
+O(z^3\bar\xi^2)+O(z^2\bar\xi^4)+O(\bar\xi^5),
\end{multline}
where we substitute $w=\bar\eta e^{i\bar\eta^{m-1}\varphi(z,\bar\xi,\bar\eta)}$. We then gather in \eqref{huge} terms with $z^0\bar\xi^2$ and find 

$$a_0(\bar\eta)=\bar\eta^{m-1}- 2i\varphi_{22}(\bar\eta).$$ 

Gathering terms with, respectively, $z^1 \bar \xi^2,\,z^0\bar\xi^3\,z^0\bar\xi^4\,z^2\bar\xi^2$, we get:

$$a_1(\bar \eta)=-6i \varphi_{32} (\bar\eta),\quad b_0(\bar\eta)=-2\varphi_{23}(\bar \eta),\quad c_0(\bar\eta)=2i\varphi_{24}(\bar \eta),\quad a_2(\bar\eta)=-12i\varphi_{42}(\bar \eta).$$ 

Next, using the above formula for $a_0$ we gather terms with $z^3 \bar \xi$ and find

$$b_1(\bar\eta)= -6\varphi_{33}(\bar\eta) +2i(m-1)\varphi_{22}(\bar\eta)\bar\eta^{m-1}-8(\varphi_{22}(\bar\eta))^2 -
2i\varphi'_{22}(\bar\eta)\bar\eta^m.$$

Finally, using the above formulas for $a_0,b_0,b_1$ and gathering terms with, respectively, $z^2\bar\xi^3$ and $z^1\bar\xi^4$, we get first

$$b_2(\bar\eta)=-12\varphi_{43}(\bar\eta)+36\varphi_{22}(\bar\eta)\varphi_{32}(\bar\eta)+6i(1-m)\bar\eta^{m-1}\varphi_{32}(\bar\eta)+6i\bar\eta^m\varphi_{32}\rq{}(\bar\eta)$$
and second 
$$c_1(\bar\eta)=6i\varphi_{34}(\bar\eta)-20i\varphi_{22}(\bar\eta)\varphi_{23}(\bar\eta)+(4-4m)\bar\eta^{m-1}\varphi_{23}(\bar\eta)+2\bar\eta^m\varphi_{23}\rq{}(\bar\eta).$$

Now the Fuchsian type conditions \eqref{fuchsianity} yield 
\begin{equation}\label{fuchsianityODE1}
\begin{aligned}
\mbox{ord}\,a_0&\geq \,\,\,m-1,\quad \mbox{ord}\,a_1\geq \,\,m-1,\quad \mbox{ord}\,a_{2}\geq \,m-1,\\\mbox{ord}\,b_0&\geq 2m-2,\quad \mbox{ord}\,b_1\geq 2m-2,\quad \mbox{ord}\,b_{2}\geq 2m-2,\\
\quad  \mbox{ord}\,c_0&\geq 3m-3,\quad \mbox{ord}\,c_1\geq 3m-3,
\end{aligned}
\end{equation}
which is already equivalent to \eqref{fuchsianityODE}, as required. The proof in the negative case is analogues.

\end{proof}

\subsection{Analyticity in the Fuchsian case}
Let now $M$, given as in \eqref{madmissible}, satisfies the Fuchsianity conditions \eqref{fuchsianity},\eqref{fuchsianityODE}, and $L\in\hol{\infty}(M,0)$. According to the previous section, we have
\begin{equation}\label{simplefield}
P(z,w)=P_0(w)+P_1(w)z+q_1\rq{}(w)z^2-2q_1(w)\tilde a(z,w), \quad Q(z,w)=Q_0(w)+Q_1(w)z.
\end{equation}
Here $P_j$ and $Q_j$ are $C^\infty$ functions on the real line, admitting a holomorphic extension to sectors $S^\pm$, as in \eqref{sectors}. Introduce the following notation: for functions $f_1(w),..,f_n(w)$, $\mathcal L(f)$ denotes their finite linear combination with coefficients being holomorphic at the origin functions of $w$. Then, gathering in the last two equations in \eqref{initial} terms of degrees $0$ and $1$ in $z$ and taking the fuchsianity conditions into account, we obtain:   

\begin{equation}\label{bigsystem}
\begin{aligned}
Q_0\rq{}\rq{}&=\mathcal L(P_1\rq{})+\frac{1}{w}\mathcal L(P_0,P_1,Q_0\rq{})+\frac{1}{w^2}\mathcal L(Q_0,Q_1)\\
Q_1\rq{}\rq{}&=\frac{1}{w}\mathcal L(P_0,P_1,Q_0\rq{},Q_1\rq{})+\frac{1}{w^2}\mathcal L(Q_0,Q_1)\\
P_0\rq{}\rq{}&=\frac{1}{w}\mathcal L(P_0\rq{})+\frac{1}{w^2}\mathcal L(P_0,P_1,Q_0\rq{})+\frac{1}{w^3}\mathcal L(Q_0,Q_1)\\
P_1\rq{}\rq{}&=\frac{1}{w}\mathcal L(P_0\rq{},P_1\rq{})+\frac{1}{w^2}\mathcal L(P_0,P_1,Q_0\rq{},Q_1\rq{})+\frac{1}{w^3}\mathcal L(Q_0,Q_1).
\end{aligned}
\end{equation}
We next need the following simple
\begin{lemma}\label{factorw}
There exists a $C^\infty$ CR-function $R$ on $M$ such that $Q(z,w)=wR(z,w)$.
\end{lemma}
\begin{proof}
Since the complex locus $X=\{w=0\}$ is a CR-orbit of $M$, it is invariant under CR-diffeomorphisms, so that for its flow $(z,w)\mapsto (F_t(z,w),G_t(z,w))$  all $G_t(z,w)$ vanish identically for $w=0$. Hence $Q(z,w)$ vanishes identically for $w=0$. We conclude that $Q(z,w)$ is divisible by $w$ (remaining $C^\infty$ smooth), as required.
\end{proof}

Using \autoref{factorw}, we easily rewrite \eqref{bigsystem} as 
\begin{equation}\label{bigsystem1}
\begin{aligned}
R_0\rq{}\rq{}&=\frac{1}{w}\mathcal L(P_1\rq{},R_0\rq{})+\frac{1}{w^2}\mathcal L(P_0,P_1,R_0,R_1)\\
R_1\rq{}\rq{}&=\frac{1}{w}\mathcal L(R_0\rq{},R_1\rq{})+\frac{1}{w^2}\mathcal L(P_0,P_1,R_0,R_1)\\
P_0\rq{}\rq{}&=\frac{1}{w}\mathcal L(P_0\rq{},R_0\rq{})+\frac{1}{w^2}\mathcal L(P_0,P_1,R_0,R_1)\\
P_1\rq{}\rq{}&=\frac{1}{w}\mathcal L(P_0\rq{},P_1\rq{},R_0\rq{},R_1\rq{})+\frac{1}{w^2}\mathcal L(P_0,P_1,R_0,R_1),
\end{aligned}
\end{equation}
where $Q=w\cdot R$ for a $C^\infty$ CR-function $R$ on $M$.
Let us introduce the vector function
$$Y(w):\quad\CC{2}\lr \CC{8},\quad Y:=(P_0,P_1,R_0,R_1,wP_0\rq{},wP_1\rq{},wR_0\rq{},wR_1\rq{}).$$
Taking into account the relations between the components of the vector $Y$, it is not difficult to make sure that \eqref{bigsystem1} turns into a system of first order {\em Fuchsian} linear complex ODEs 
\begin{equation}\label{completefuchsian}
\frac {d Y}{dz}=\frac{1}{w}AY,
\end{equation}
where $A(w)$ is a holomorphic at the origin $8\times 8$ matrix functions. We are now in the position to prove our second result.
\begin{proof}[Proof of \autoref{analyticity}] Let us first consider the case when $(M\setminus X)\cap U$ is Levi-nondegenerate. Then we may assume that $M$, given as in \eqref{madmissible}, satisfies the Fuchsianity conditions \eqref{fuchsianity},\eqref{fuchsianityODE}. We then use the above argument to obtain \eqref{completefuchsian}. Since the system of ODEs \eqref{completefuchsian} is Fuchsian, we have (e.g. \cite{ilyashenko}) 
\begin{equation}\label{monodromyformula}
Y(w)=H(w)\cdot w^\Lambda\cdot \tau,
\end{equation} 
where $\Lambda$  and $\tau$ are  constant $8\times 8$ and $8\times 1$ matrices respectively, and $H(w)$ is a holomorphic at the origin $8\times 8$ matrix function. (In fact,  the matrix $\Lambda$ here equals $\frac{1}{2\pi i}\ln G$, where $G$ is the monodromy matrix of \eqref{completefuchsian} near the singularity $w=0$). Recall that all $P_j,Q_j$ admit an asymptotic expansion in sectors $S^\pm$, as in \eqref{sectors}. In view of \eqref{monodromyformula}, the latter fact implies that $Y(w)$ extends to the origin holomorphically, and so does the vector field $L$, as required.

If, otherwise, the reference point $0\in X\cap\Sigma_1$, we use the existence of a blow up map \eqref{blowup} (\autoref{blowuplemma}) for which $M$ pulls back to a hypersurface of Fuchsian type at the origin. Via the map $(z,w) = (\xi^s \eta, \eta^2),\,s\geq 2$, the hypersurface
$M$ pulls back to a hypersurface $M^*$ such that $M^*\setminus\{\eta=0\}$ is  Levi-nondegenerate, and $L$ to a vector field 
$$L^*=P^*(\xi,\eta)\frac{\partial}{\partial \xi}+Q^*(\xi,\eta)\frac{\partial}{\partial \eta}\in\hol{}(M^*,q^*)$$
for some $q^*\in M^*\setminus\{\eta=0\}$, where 
\begin{equation}\label{newfield}
P^*(\xi,\eta)=\frac{1}{\eta^s}P(\xi\eta^s,\eta^2)-\frac{s\xi}{2\eta^2}Q(\xi\eta^s,\eta^2), \quad
Q^*(\xi,\eta)=\frac{1}{2\eta}Q(\xi\eta^s,\eta^2).
\end{equation}

Repeating word-by-word the argument in the first part of the proof for the case of a vector field of the form  
  \[\tilde L = \frac{f(z,w)}{w^n} \dz + \frac{g(z,w)}{w^n} \dw , \] 
  which is an infinitesimal automorphism of $M$ on $M\setminus X$
   with $f$ and $g$ being smooth CR functions on $M$, we get that $\eta^s P^*,\eta Q^*$ extend to the origin holomorphically. 
We conclude that 
\begin{equation}\label{ff}
P(\xi\eta^s,\eta^2)=f(\xi,\eta)=\sum_{j\geq 0} f_j(\eta)\xi^j, \quad 
Q(\xi\eta^s,\eta^2)=g(\xi,\eta)=\sum_{j\geq 0} g_j(\eta)\xi^j.
\end{equation} 
for some holomorphic near  the origin $f,g$. We then note that, since $L$ is an $C^\infty$ CR-automorphism, all $P_{z^k}(0,w),\,k\geq 0$ are $C^\infty$ functions on the real line, and similarly for $Q$. Hence, differentiating both equations \eqref{ff} $k$ times in $\xi$  for each $k\geq 0$ and substituting $\xi=0$,  we get that all $f_j$ are divisible by $\eta^{js}$, and similarly for $g_j$.  It is also easy to see that all $f_j(\eta)/\eta^{js}=:P_j(\eta^2)$ are in turn asymptotic power series in $\eta^2$. The latter facts imply that $P(z,w)=\sum_{j\geq 0} P_j(w)z^j,\,Q(z,w)=\sum_{j\geq 0} Q_j(w)z^j$ are holomorphic near the origin functions, as required.

\end{proof}

\subsection{Convergence of formal automorphisms} Let $M$, given as in \eqref{madmissible}, satisfies the Fuchsianity conditions \eqref{fuchsianity},\eqref{fuchsianityODE}, and $L\in\hol{f}(M,0)$. The proof of the fact that $L$ is convergent is very similar to the proof of the analyticity in the $C^\infty$ case, that is why we leave the details to the reader and provide below a scheme of the proof of \autoref{convergence} only.

\smallskip

- We first consider the case when $M\setminus X$ is Levi-nondegenerate and transfer to the coordinates \eqref{madmissible}.

\smallskip

- Let us introduce the special 2-jet space with the coordinates $(z,w,\zeta,w_2)$, where $w_2$ corresponds to the second derivative $w''$ over  a curve in $\CC{2}$, and $\zeta$ corresponds to $w'/w^m$. We then get a special 2-jet prolongation of $M$, obtained by substituting $w_1:=\zeta w^m$ into \eqref{prolong}. Note that the result is a formal meromorphic vector field. 

\smallskip

- Since the flow of $L$ preserves the Segre family of $M$,  the above special 2-jet prolongation is tangent to the submanifold 
$$w_2=\Phi(z,w,\zeta)$$
of the special 2-jet space,
corresponding to the associated ODE \eqref{ODE}.

\smallskip

- Accordingly, we collect terms with $\zeta^0, \zeta^1,\zeta^2,\zeta^3$ in the tangency condition and get the equations \eqref{initial}. Now arguments identical to the $C^\infty$ case yield the Fuchsian system \eqref{completefuchsian} (with the same notations), with $A$ being holomorphic at the origin. 

\smallskip

- As a well-known fact (e.g. \cite{vazow}), all formal power series solutions of \eqref{completefuchsian} are convergent. Hence $Y(w)$ is convergent, and so is $L$. This proves \autoref{convergence} is the case when $M\setminus X$ is Levi-nondegenerate. 

\smallskip

- Otherwise, we apply a blow up map \eqref{blowup} and argue then identically to the proof of \autoref{analyticity} in the case when $M\setminus X$ contains Levi-degenerate points. This completely proves \autoref{convergence}. 

\qed

\end{document}